\documentclass{amsart}

\usepackage{enumerate}
\usepackage{amsmath}
\usepackage{amssymb}
\usepackage{amsthm}

\usepackage{math}
\usepackage{thm}
\usepackage{ref}

\title
	[On Two-Phase Flows with Soluble Surfactant]
	{On Two-Phase Flows with Soluble Surfactant}

\author
	[Dieter Bothe]
	{Dieter Bothe}
	
\address
	{Center of Smart Interfaces and\newline\indent
	 Department of Mathematics\newline\indent
	 Technische Universit{\"a}t Darmstadt, \newline\indent
	 Petersenstr.~32, D-64287 Darmstadt, Germany}
	
\email
	{bothe@csi.tu-darmstadt.de}

\author
	[Matthias K{\"o}hne]
	{Matthias K{\"o}hne}

\address
	{Department of Mathematics\newline\indent
	 Technische Universit{\"a}t Darmstadt, \newline\indent
	 Schlossgartenstr.~7, D-64289 Darmstadt, Germany}
	
\email
	{koehne@csi.tu-darmstadt.de}

\author
	[Jan Pr{\"u}ss]
	{Jan Pr{\"u}ss}

\address
	{Institut f{\"u}r Mathematik \newline\indent
	 Martin-Luther-Universit{\"a}t Halle-Wittenberg, \newline\indent
	 Theodor-Lieser-Str.~5, D-60120 Halle, Germany}

\email
	{jan.pruess@mathematik.uni-halle.de}

\keywords
	{two-phase Navier-Stokes equations,
	 surface tension,
	 surfactant,
	 Marangoni effects,
	 local-in-time well-posedness,
	 maximal regularity,
	 initial boundary value problem,
	 free boundary problem}

\subjclass
	[2010]
	{Primary: 35Q30; Secondary: 35R35, 76D03, 76D45}

\date
	{\today}

\begin{document}
\begin{abstract}
	The presence of surfactants has a pronounced effect on the surface tension and, hence, on the stress balance at the
	phase separating interface of two-phase flows.
	The transport of momentum induced by the local variations of the capillary forces are known as {\itshape Marangoni effects}.
	Here we study a model, which assumes the surfactant to be soluble in one of the adjacent bulk phases
	and which represents a generalization of the two-phase Navier-Stokes equations.
	Based on maximal $L_p$-regularity results for suitable linearizations
	we obtain local well-posedness of this model.
	We employ recent results from the $L_p$-theory of two-phase flows without surfactant.
\end{abstract}
\renewcommand{\baselinestretch}{1.125}
\normalsize
\maketitle

\section*{Introduction}
We consider a free boundary problem, which describes the isothermal flow of two immiscible, incompressible Newtonian fluids.
To be precise, we assume the fluids to have constant densities $\rho_\pm > 0$ and constant viscosities $\eta_\pm > 0$.
They occupy the bounded domain $\Omega \subseteq \bR^n$, where the dispersed phase is located in $\Omega_-(t) \subseteq \Omega$
and separated from the continuous phase, which is located in $\Omega_+(t) \subseteq \Omega$, by a sharp interface $\Gamma(t)$.
The phase configuration is completely determined by the position of the interface, which, however, is time dependent and has to
be resolved as part of the problem.
The evolution of the velocity field $u$ and the pressure $p$ is governed by the Navier-Stokes equations
\begin{subequations}
	\eqnlabel{model}
\begin{equation}
	\eqnlabel{model:fluid}
	\begin{array}{rcll}
		\rho D^u_t u - \eta \Delta u + \nabla p & = & \rho f, & \qquad t > 0,\ x \in \Omega_\pm(t),   \\[0.5em]
		                          \mbox{div}\,u & = & 0,      & \qquad t > 0,\ x \in \Omega_\pm(t),   \\[0.5em]
		                  [u]_{\partial \Omega} & = & 0,      & \qquad t > 0,\ x \in \partial \Omega, \\[0.5em]
		                                   u(0) & = & u_0,    & \qquad x \in \Omega_\pm(0),
	\end{array}
\end{equation}
where we denote by $D^u_t = \partial_t + u \cdot \nabla$ the \emph{material derivative} w.\,r.\,t.\ $u$
and $[\,\cdot\,]_{\partial \Omega}$ denotes the trace of a quantity defined in $\Omega_+(t)$ on the boundary $\partial \Omega$.
As above it will be convenient to drop the phase subscripts and to write $\rho,\,\eta,\,\dots$ instead of $\rho_\pm,\,\eta_\pm,\,\dots$,
whenever there is no danger of confusion.

We assume, that no phase transitions and no interfacial slip occur, which implies the velocity field to be continuous across the interface.
In this case the normal velocity $V_\Gamma$ of the interface equals the normal component of the velocity field.
In summary, the initial phase configuration and the transmission conditions for mass and momentum
\begin{equation}
	\eqnlabel{model:interface}
	\begin{array}{rcll}
		                                                        [\![u]\!]_{\Gamma(t)} & = & 0,                                                    & \qquad t > 0,\ x \in \Gamma(t), \\[0.5em]
		-[\![\eta (\nabla u + \nabla u^{\sf{T}}) - p]\!]_{\Gamma(t)}\,\nu_{\Gamma(t)} & = & \mbox{div}_{\Gamma(t)}\,\{\,\sigma P_{\Gamma(t)}\,\}, & \qquad t > 0,\ x \in \Gamma(t), \\[0.5em]
		                                                                     V_\Gamma & = & [u]_{\Gamma(t)} \cdot \nu_{\Gamma(t)},                & \qquad t > 0,\ x \in \Gamma(t), \\[0.5em]
		                                                                    \Gamma(0) & = & \Gamma_0
	\end{array}
\end{equation}
completely determine the evolution of the interface.
Here, $\nu_{\Gamma(t)}$ denotes the normal field on $\Gamma(t)$ pointing from $\Omega_-(t)$ into $\Omega_+(t)$.
Moreover, $P_{\Gamma(t)} = 1 - \nu_{\Gamma(t)} \otimes \nu_{\Gamma(t)}$ denotes the projection onto the tangent space of $\Gamma(t)$
and $[\,\cdot\,]_{\Gamma(t)}$ denotes the trace on $\Gamma(t)$ of a quantity defined in $\Omega_\pm(t)$.
Finally, $[\![\,\cdot\,]\!]_{\Gamma(t)}$ denotes the jump of a quantity defined in $\Omega_\pm(t)$ across $\Gamma(t)$, i.\,e.
\begin{equation*}
	\begin{array}{l}
		[\![\phi]\!]_{\Gamma(t)}\,(t,\,x) = \\[0.5em]
			\quad {\displaystyle{\lim_{\epsilon \rightarrow 0+}}} \big\{\,\phi(t,\,x + \epsilon \nu_{\Gamma(t)}(t,\,x)) - \phi(t,\,x - \epsilon \nu_{\Gamma(t)}(t,\,x))\,\big\},
			\quad t > 0,\ x \in \Gamma(t).
	\end{array}
\end{equation*}

If a surface active agent -- \emph{surfactant} for short -- is present,
then the surface tension $\sigma > 0$ depends on the surface specific concentration of this surfactant
via a so-called equation of state
\begin{equation}
	\eqnlabel{model:surface-tension}
	\sigma = \sigma(c_\Gamma).
\end{equation}
Hence, we have
\begin{equation*}
	\mbox{div}_{\Gamma(t)}\,\{\,\sigma(c_\Gamma) P_{\Gamma(t)}\,\} = \sigma(c_\Gamma)\,\kappa_{\Gamma(t)}\,\nu_{\Gamma(t)} + \sigma^\prime(c_\Gamma) \nabla_{\Gamma(t)}\,c_\Gamma,
\end{equation*}
where $\kappa_{\Gamma(t)}$ denotes the sum of the principal curvatures of the interface.
Note that the surfacant is adsorbed at the interface $\Gamma(t)$ in this case and we denote by $c_\Gamma$ its surface specific concentration.

Moreover, we assume such a surfactant to be soluble in the continuous phase $\Omega_+(t)$ and denote by $c$ its volume specific concentration.
Based on Ficks law of diffusion the evolution of the surfactant is governed by
\begin{equation}
	\eqnlabel{model:surfactant}
	\begin{array}{rcll}
		                                                                       D^u_t c - d \Delta c & = & 0,                  & \qquad t > 0,\ x \in \Omega_+(t),     \\[0.5em]
		                                                                    \alpha([c]_{\Gamma(t)}) & = & c_\Gamma,           & \qquad t > 0,\ x \in \Gamma(t),       \\[0.5em]
		D^u_t c_\Gamma + c_\Gamma\,\mbox{div}_{\Gamma(t)}\,u - d_\Gamma \Delta_{\Gamma(t)} c_\Gamma & = & d \partial^+_\nu c, & \qquad t > 0,\ x \in \Gamma(t),       \\[0.5em]
		                                                                             \partial_\nu c & = & 0                   & \qquad t > 0,\ x \in \partial \Omega, \\[0.5em]
		                                                                                       c(0) & = & c_0,                & \qquad x \in \Omega_+(0),
	\end{array}
\end{equation}
where the normal derivatives have to be understood as $\partial^\pm_\nu = \mp [\nabla\,\cdot\,^{\sf{T}}]_{\Gamma(t)}\,\nu_{\Gamma(t)}$
resp.\ $\partial_\nu = [\nabla\,\cdot\,^{\sf{T}}]_{\partial \Omega}\,\nu_{\partial \Omega}$.
Note that $c_\Gamma$ is defined only on the graph $\mbox{gr}(\Gamma(t))$.
However, due to the kinematic condition $V_\Gamma = [u]_{\Gamma(t)} \cdot \nu_{\Gamma(t)}$ its material derivative is well-defined as
\begin{equation*}
	D^u_t c_\Gamma(t,\,x) := {\left. \frac{\mbox{d}}{\mbox{d}s} c_\Gamma(s,\,\chi^u(s\,|\,t,\,x)) \right|}_{s = t}, \qquad t > 0,\ x \in \Gamma(t),
\end{equation*}
where $\chi^u(\,\cdot\,|\,t,\,x)$ denotes the characteristic curve of a particle, which at time $t > 0$ is located at $x \in \Omega$
and which is advected by the velocity field $u$, i.\,e.
\begin{equation*}
	\dot{\chi}^u(s\,|\,t,\,x) = u(s,\,\chi^u(s\,|\,t,\,x)), \quad |s - t| < \epsilon, \qquad \chi^u(t\,|\,t,\,x) = x.
\end{equation*}
In fact, $x \in \Gamma(t)$ implies $\chi^u(s\,|\,t,\,x) \subseteq \Gamma(s)$, if $|s - t| < \epsilon$.
This is the reason why we prefer the notation employing the material derivative even for the momentum balance in \eqnref{model:fluid}
and the surfactant diffusion equation in the bulk phase in \eqnref{model:surfactant}.
\end{subequations}

\section{An Analytic Approach}
	\seclabel{approach}
The two-phase Navier-Stokes equations with surface tension, i.\,e.\ (\eqnref*{model:fluid},\,\eqnref*{model:interface}) with constant $\sigma > 0$,
have already been intensively studied.
For the most recent results concerning local-in-time well-posedness as well as the qualitative behaviour of solutions
we refer to \cite{Koehne-Pruess-Wilke:Two-Phase-Navier-Stokes, Shibata-Shimizu:Two-Phase-Navier-Stokes} and the references therein.
The two-phase Navier-Stokes equations with soluble surfactant \eqnref{model} have been studied in a prototype geometry,
where the interface is almost flat, see \cite{Bothe-Pruess-Simonett:Surfactant-Halfspace}.
The first aim of the present paper is to transfer this result to a general geometry, i.\,e.\ we will prove
\begin{theorem}
	\thmlabel{well-posedness}
	Let $\Omega \subseteq \bR^n$ be a bounded domain with boundary of class $C^{3-}$ and let $p > n + 2$.
	Let $\rho_\pm,\,\eta_\pm,\,d,\,d_\Gamma > 0$ and let $\sigma,\,\alpha \in C^{3-}(\bR_+,\,\bR_+)$ with $\alpha^\prime > 0$.
	Suppose
	\begin{equation*}
		u_0 \in W^{2 - 2/p}_p(\Omega \setminus \Gamma_0), \qquad \Gamma_0 \in W^{3 - 2/p}_p, \qquad c_0 \in W^{2 - 2/p}_p(\Omega_+(0),\,\bR_+)
	\end{equation*}
	are subject to the regularity and compatibility conditions
	\begin{equation*}
		\begin{array}{c}
			\mbox{\upshape div}\,u_0 = 0 \quad \mbox{\upshape in}\ \Omega \setminus \Gamma_0, \qquad [\![u_0]\!]_{\Gamma_0} = 0, \qquad [u_0]_{\partial \Omega} = 0, \\[0.5em]
			- P_{\Gamma_0} [\![\eta (\nabla u + \nabla u^{\sf{T}})]\!]_{\Gamma_0}\,\nu_{\Gamma_0} = \nabla_{\Gamma_0}\,(\sigma \circ \alpha \circ [c_0]_{\Gamma_0}), \\[0.5em]
			[c_0]_{\Gamma_0} \in W^{2 - 2/p}_p(\Gamma_0), \qquad \partial_\nu c_0 = 0 \quad \mbox{\upshape on}\ \partial \Omega
		\end{array}
	\end{equation*}
	and $f \in L_p(\bR_+,\,L_p(\Omega,\,\bR^n))$.
	Then there exists $a_0 = a_0(u_0,\,\Gamma_0,\,c_0) > 0$,
	such that the two-phase Navier-Stokes equations with soluble surfactant \eqnref{model} admit a unique local strong solution $(u,\,p,\,\Gamma,\,c,\,c_\Gamma)$
	with $c,\,c_\Gamma > 0$ in $(0,\,a_0)$.

	If in addition $\sigma,\,\alpha \in C^\omega(\bR_+,\,\bR_+)$ and $f = 0$, then the above solution is a classical solution.
	The graph
	\begin{equation*}
		\mbox{\upshape gr}\,\Gamma = \bigcup_{0 < t < a_0} \{\,t\,\} \times \Gamma(t)
	\end{equation*}
	is a real analytic manifold and with
	\begin{equation*}
		\mho   = \{\,(t,\,x) \in (0,\,a_0) \times \Omega\,:\,x \notin \Gamma(t)\,\}, \
		\mho_+ = \{\,(t,\,x) \in (0,\,a_0) \times \Omega\,:\,x \in \Omega_+(t)\,\}
	\end{equation*}
	the functions $(u,\,p): \mho \longrightarrow \bR^{n + 1}$ and $c: \mho_+ \longrightarrow \bR_+$ are real analytic.
\end{theorem}

The proof of \Thmref{well-posedness} will be carried out in \Secsref{well-posedness}{analyticity}.
However, besides the local well-posedness of model \eqnref{model} it is worthwhile to study its dynamics,
i.\,e.\ to characterize the equilibria and to study the qualitative behaviour of the solutions.
The state of the system is described in terms of the velocity, the position of the interface and the surfactant concentration in the bulk phase.
To determine the corresponding phase manifold, we denote by $\cM\cH^2(\Omega)$ the manifold of all compact $C^2$-hypersurfaces in $\Omega$,
which becomes a metric space when equipped with the Hausdorff metric
\begin{equation*}
	\mbox{dist}_{\cM\cH^2(\Omega)}(\Sigma,\,\Gamma) := \mbox{dist}(N^2 \Sigma,\,N^2 \Gamma), \qquad \Sigma,\,\Gamma \in \cM\cH^2(\Omega).
\end{equation*}
Here, $N^2 \Sigma$ denotes the second normal bundle of a hypersurface $\Sigma \in \cM\cH^2(\Omega)$.
As has been noted in \cite[Section 5]{Koehne-Pruess-Wilke:Two-Phase-Navier-Stokes},
every hypersurface $\Sigma \in \cM\cH^2(\Omega)$ may be described by a level-set function $\varphi_\Sigma \in C^2(\bar{\Omega})$, i.\,e.\ $\Sigma = \varphi^{-1}_\Sigma(0)$,
and, of course, $\Sigma$ is of class $W^s_p$, if and only if $\varphi_\Sigma \in W^s_p(\Omega)$.
With this notation, the phase manifold of the dynamical system \eqnref{model} is given by
\begin{equation}
	\eqnlabel{phase-manifold}
	\cS_p(\Omega) := \left\{\begin{array}{c} (v,\,\Sigma,\,\chi) \in C(\bar{\Omega})^n \times \cM\cH^2(\Omega) \times \dot{C}(\bar{\Omega}) \\[0.5em] \hline \\[-0.75em] v \in W^{2 - 2/p}_p(\Omega),\ \Sigma \in W^{3 - 2/p}_p, \\[0.5em] \mbox{div}\,v = 0\ \mbox{in}\ \Omega \setminus \Sigma,\ [\![v]\!]_\Sigma = 0,\ [v]_{\partial \Omega} = 0, \\[0.5em] \mbox{dom}\,\chi = \Omega \setminus \mbox{int}(\Sigma),\ \chi \in W^{2 - 2/p}_p(\Omega \setminus \mbox{int}(\Sigma)), \\[0.5em] [\chi]_\Sigma \in W^{2 - 2/p}_p(\Sigma),\ \partial_\nu \chi = 0\ \mbox{on}\ \partial \Omega \\[0.5em] - P_\Sigma [\![\eta (\nabla v + \nabla v^{\sf{T}})]\!]_\Sigma\,\nu_\Sigma = \sigma^\prime(\alpha([\chi]_\Sigma))\nabla_\Sigma\alpha([\chi]_\Sigma) \end{array}\right\},
\end{equation}
where we denote by
\begin{equation*}
	\dot{C}(\bar{\Omega}) := \left\{\,\phi \in C(\bar{U})\,:\,U \subseteq \bR^n\ \mbox{open},\ U \subseteq \Omega\,\right\}
\end{equation*}
the set of functions, which are defined and continuous on certain subsets of $\bar{\Omega}$,
and $\mbox{dom}\,\phi$ denotes the domain of definition of such a function $\phi$.
Moreover, for $\Sigma \in \cM\cH^2(\Omega)$ we denote by $\mbox{int}(\Sigma)$ the interior of $\Sigma$, i.\,e.\ the union of all connected components of $\Omega \setminus \Sigma$,
which are not in contact with $\partial \Omega$.
Since the solutions to \eqnref{model}, which exist thanks to \Thmref{well-posedness}, preserve the compatibility conditions of the phase manifold,
we may prove
\begin{theorem}
	\thmlabel{semiflow}
	Let $\Omega \subseteq \bR^n$ be a bounded domain with boundary of class $C^{3-}$ and let $p > n + 2$.
	Let $\rho_\pm,\,\eta_\pm,\,d,\,d_\Gamma > 0$ and let $\sigma,\,\alpha \in C^{3-}(\bR_+,\,\bR_+)$ with $\alpha^\prime > 0$.
	Then the strong solutions to \eqnref{model} generate a local semiflow in the phase manifold $\cS_p(\Omega)$.
	Each of these solutions exists on a maximal time interval $[0,\,a^\ast)$.
\end{theorem}
The proof of \Thmref{semiflow} will be carried out in \Secref{semiflow}.
Note that neither the pressure nor the surfactant concentration on the interface explicitly appear in the definition of the phase manifold $\cS_p(\Omega)$.
In fact, the surfactant concentration on the interface is well-defined for every state $(v,\,\Sigma,\,\chi) \in \cS_p(\Omega)$ as $\chi_\Gamma = \alpha([\chi]_\Sigma)$.
Moreover, the pressure may be reconstructed for every semiflow
\begin{equation*}
	(u,\,\Gamma,\,c) \in BUC([0,\,a_0),\,\cS_p(\Omega)),
\end{equation*}
which is induced by a strong solution to \eqnref{model}, i.\,e.\ which in particular satisfies $\partial_t u,\,\Delta u \in L_p((0,\,a_0) \times \Omega)^n$.
Indeed, in this case we have
\begin{equation*}
	(\,u\,|\,\nabla \phi\,)_\Omega = -\,(\,\mbox{div}\,u\,|\,\phi\,)_\Omega - (\,[\![u]\!]_\Gamma \cdot \nu_\Gamma\,|\,\phi\,)_\Gamma = 0,
	\qquad \phi \in H^1_{p^\prime}(\Omega),
\end{equation*}
which implies $(\,\partial_t u\,|\,\nabla \phi\,)_\Omega = 0$ for all $\phi \in H^1_{p^\prime}(\Omega)$, where $1/p + 1/p^\prime = 1$.
Hence, the momentum balance in \eqnref{model:fluid} and the momentum transmission condition in \eqnref{model:interface} yield
\begin{equation*}
	\begin{array}{rcll}
		(\,\nabla \pi\,|\,\nabla \phi\,)_\Omega & = & (\,\frac{\mu}{\rho} \Delta u - u \cdot \nabla u\,|\,\nabla \phi\,)_\Omega,                                                             & \quad \phi \in H^1_{p^\prime}(\Omega), \\[0.5em]
		           [\![\rho \pi]\!]_{\Gamma(t)} & = & \sigma(c_\Gamma)\,\kappa_{\Gamma(t)} + [\![\eta (\nabla u + \nabla u^{\sf{T}})]\!]_{\Gamma(t)}\,\nu_{\Gamma(t)} \cdot \nu_{\Gamma(t)}, & \quad t > 0,\ x \in \Gamma(t)
	\end{array}
\end{equation*}
for the modified pressure $\pi = p / \rho$, where, of course, $c_\Gamma = \alpha([c]_\Gamma)$.
Thanks to \cite[Theorem~8.5]{Koehne-Pruess-Wilke:Two-Phase-Navier-Stokes}, this problem has a unique solution $\pi \in L_p((0,\,a_0) \times \Omega)$,
such that $\pi(t) \in \hat{W}^{1 - 1/p}_p(\Omega \setminus \Gamma(t))$ for $0 < t < a_0$.

As has been worked out in \cite{Bothe-Pruess:Surfactant-Stability}, the equilibria are closely related to the energy functional
\begin{subequations}
	\eqnlabel{energy}
\begin{equation}
	\eqnlabel{energy:lyapunov}
	\Phi(t\,|\,u,\,\Gamma,\,c) = \frac{1}{2} \|\sqrt{\rho} u(t)\|^2_{L_2(\Omega,\,\bR^n)} + \!\!\! \int_{\Omega_+(t)} \!\!\! \phi(c)\,\mbox{d}\cH^n + \int_{\Gamma(t)} \phi_\Gamma(\alpha([c]_{\Gamma(t)}))\,\mbox{d}\cH^{n - 1},
\end{equation}
which is composed of the kinetic energy and the free energies in the bulk phase and on the interface.
The latter are obtained via the energy densities $\phi$ and $\phi_\Gamma$, which may be obtained as
\begin{equation}
	\eqnlabel{energy:densities}
	\phi(s) = \int^s_0 \mu_\Gamma(\alpha(r))\,\mbox{d}r, \quad \phi_\Gamma(s) = \sigma(s) + s \mu_\Gamma(s), \quad - \mu_\Gamma(s) = \int^s_0 \sigma^\prime(r) / r\,\mbox{d}r
\end{equation}
\end{subequations}
based on the chemical potential $\mu_\Gamma$.
With these definitions \cite[Theorem~3.1]{Bothe-Pruess:Surfactant-Stability} states the following.
\begin{theorem}
	\thmlabel{equilibria}
	Let $\Omega \subseteq \bR^n$ be a bounded domain with boundary of class $C^{3-}$ and let $p > n + 2$.
	Let $\rho_\pm,\,\eta_\pm,\,d,\,d_\Gamma > 0$ and let $\sigma,\,\alpha \in C^{3-}(\bR_+,\,\bR_+)$ with $- \sigma^\prime,\,\alpha^\prime > 0$.
	Let the energy functional $\Phi$ and the energy densities $\phi,\,\phi_\Gamma$ be defined by \eqnref{energy}.
	Then we have
	\begin{enumerate}[1.]
		\item The energy equality
			\begin{equation*}
				\partial_t \Phi + 2 \eta \!\!\!\! \int_{\Omega \setminus \Gamma(t)} \!\!\!\! |D|^2\,\mbox{\upshape d}\cH^n
					+ d \!\!\! \int_{\Omega_+(t)} \!\!\! \phi^{\prime \prime}(c)\,|\nabla c|^2\,\mbox{\upshape d}\cH^n
					+ d_\Gamma \!\! \int_{\Gamma(t)} \!\! \phi^{\prime \prime}_\Gamma(c_\Gamma)\,|\nabla_\Gamma c_\Gamma|^2\,\mbox{\upshape d}\cH^{n - 1}
					= 0
			\end{equation*}
			is valid for smooth solutions to \eqnref{model}.
		\item The equilibria of \eqnref{model} are zero velocities, constant pressures in the connected components of the two phases,
			constant surfactant concentrations in $\Omega_+$ and on $\Gamma$ and the dispersed phase is a union of non-intersecting balls.
		\item The energy functional $\Phi$ is a strict Lyapunov functional.
		\item The critical points of the energy functional $\Phi$ for prescribed volumes of the connected components of the dispersed phase
			and prescribed total surfactant mass are precisely the equilibria of the system. \qed
	\end{enumerate}
\end{theorem}

Note that \eqnref{energy:densities} implies $\phi^{\prime \prime}(s) = -\sigma^\prime(\alpha(s)) \alpha^\prime(s) / \alpha(s)$ as well as $\phi^{\prime \prime}_\Gamma(s) = -\sigma^\prime(s) / s$ for $s > 0$,
i.\,e.\ $\phi$ and $\phi_\Gamma$ are convex, provided the surface tension $\sigma$ is monotonically decreasing
and the adsorption isotherm $\alpha$ is monotonically increasing,
which are natural assumptions.

\section{Linearization}
	\seclabel{linearization}
\subsection{The Hanzawa Transformation}
The proof of \Thmref{well-posedness} is based on maximal regularity results for suitable linearizations of \eqnref{model} and fixed point arguments.
To obtain a linearization in a fixed domain note that any hypersurface $\Sigma \in \cM\cH^2(\Omega)$ forms the boundary of the union of finitely many domains
contained in $\Omega$, i.\,e.\ $\mbox{int}(\Sigma) \subseteq \Omega$ is an open set with compact boundary of class $C^2$ and outer unit normal field $\nu_\Sigma$.
Moreover, every such hypersurface admits a tube contained in $\Omega$, i.\,e.\ there exists $\epsilon > 0$, such that the tube mapping
\begin{equation*}
	T^\epsilon_\Sigma: (-\epsilon,\,\epsilon) \times \Sigma \longrightarrow \Omega, \qquad T(r,\,x) = x + r \nu_\Sigma(x), \quad |r| < \epsilon,\ x \in \Sigma
\end{equation*}
is a $C^1$-diffeomorphism from $(-\epsilon,\,\epsilon) \times \Sigma$ onto $R(T^\epsilon_\Sigma) \subseteq \Omega$.
The supremum of all such $\epsilon > 0$ is called the tube size of $\Sigma$ and denoted by $\tau(\Sigma)$.
Note that the tube size is bounded from above by the radii of the interior and exterior ball condition of $\mbox{int}(\Sigma)$
as well as by the reciprocals of the moduli of the principal curvatures $\kappa_1,\,\dots,\,\kappa_{n - 1}$,
which are the eigenvalues of the curvature tensor $L_\Sigma = - \nabla_\Sigma\,\nu_\Sigma$.
In particular, we have
\begin{equation*}
	\tau(\Sigma) \leq \mbox{min}\,\big\{\,1 / |\kappa_i(x)|\,:\,x \in \Sigma,\ i = 1,\,\dots,\,n - 1\,\big\}.
\end{equation*}

Now, given any initial interface $\Gamma_0 \in W^{3 - 2/p}_p$ and any $\delta > 0$,
it is well-known that there exists a reference hypersurface $\Sigma \in \cM\cH^2(\Omega)$,
which may be choosen to be real analytic, such that
\begin{equation*}
	\mbox{dist}_{\cM\cH^2(\Omega)}(\Sigma,\,\Gamma_0) < \delta \quad \mbox{and} \quad \Gamma_0 \in R(T^{\epsilon/3}_\Sigma)
\end{equation*}
for some $0 < \epsilon < \tau(\Sigma)$, cf.~\cite[Section 2]{Koehne-Pruess-Wilke:Two-Phase-Navier-Stokes}.
Hence, $\Gamma_0$ may be parametrized over $\Sigma$ via a function $\gamma_0 \in W^{3 - 2/p}_p(\Sigma)$ as
\begin{equation*}
	\Gamma_0 = \big\{\,x + \gamma_0(x) \nu_\Sigma(x)\,:\,x \in \Sigma\,\big\}
\end{equation*}
and we expect the interface $\Gamma(t)$ to stay in the tube $R(T^{\epsilon/3}_\Sigma)$ and to be parametrized by a function
$\gamma(t,\,\cdot\,) \in W^{3 - 2/p}_p(\Sigma)$ as well, at least for small times $0 < t < a$.
Thus, by choosing a cut-off function $\chi \in C^\infty_0(\bR,\,[0,\,1])$ with
\begin{equation*}
	\chi(r) = 1, \quad |r| < 1 / 3, \qquad \chi(r) = 0, \quad |r| > 2 / 3, \qquad |\chi^\prime|_\infty < 4
\end{equation*}
we obtain a family of $C^1$-diffeomorphisms, the {\itshape Hanzawa transformation},
\begin{equation*}
	\begin{array}{c}
		\Theta: [0,\,a) \times \Omega \longrightarrow \Omega, \\[0.5em]
		\Theta(t,\,x) = x + \chi(d_\Sigma(x) / \epsilon)\,\gamma(t,\,\Pi_\Sigma(x))\,\nu_\Sigma(\Pi_\Sigma(x)), \quad 0 \leq t < a,\ x \in \Omega,
	\end{array}
\end{equation*}
which are of class $W^{3 - 2/p}_p$, if $\gamma \in BUC([0,\,a),\,W^{3 - 2/p}_p(\Sigma))$ and $\Sigma$ is choosen to be sufficiently smooth,
and which transform the system \eqnref{model} to a fixed domain, since
\begin{equation*}
	\Omega_-(t) = \mbox{int}(\Gamma(t)) = \Theta(t,\,\mbox{int}(\Sigma)), \quad \Omega_+(t) = \Omega \setminus \overline{\mbox{int}}(\Gamma(t)) = \Theta(t,\,\Omega \setminus \overline{\mbox{int}}(\Sigma))
\end{equation*}
and, of course, $\Gamma(t) = \Theta(t,\,\Sigma)$ for all $0 \leq t < a$.
Note, that we denote by
\begin{equation*}
	T^{-\epsilon}_\Sigma = d_\Sigma \times \Pi_\Sigma: R(T^\epsilon_\Sigma) \longrightarrow (-\epsilon,\,\epsilon) \times \Sigma
\end{equation*}
the inverse of $T^\epsilon_\Sigma$, i.\,e.\ $d_\Sigma: R(T^\epsilon_\Sigma) \longrightarrow (-\epsilon,\,\epsilon)$ is nothing but the signed distance from $\Sigma$
and $\Pi_\Sigma: R(T^\epsilon_\Sigma) \longrightarrow \Sigma$ is the metric projection onto $\Sigma$.
Also note, that the transformation $\Theta(t,\,\cdot\,)$ is well-defined for all $x \in \Omega$, even if $d_\Sigma$ and $\Pi_\Sigma$ are only defined
in $R(T^\epsilon_\Sigma)$, since the cut-off function $\chi(\,\cdot\,/ \epsilon)$ vanishes outside $(-2 \epsilon /3,\,2 \epsilon / 3)$ and $\Theta(t,\,\cdot\,)$ is to be understood
as the identity outside $R(T^{2 \epsilon /3}_\Sigma)$.

For the following computations it will be convenient, to have an alternative representation of the Hanzawa transformation at hand.
With $0 \leq t < a$ and $x \in \Sigma$ fixed we define
\begin{subequations}
	\eqnlabel{hanzawa-difference}
\begin{equation*}
	\theta(\,\cdot\,;\,t,\,x): (-\epsilon,\,\epsilon) \longrightarrow (-\epsilon,\,\epsilon),
	\qquad \theta(r\,;\,t,\,x) := r + \chi(r / \epsilon)\,\gamma(t,\,x), \quad |r| < \epsilon
\end{equation*}
and note that
\begin{equation}
	\eqnlabel{hanzawa-difference-forward}
	\Theta(t,\,x) = x + \{\,\theta(d_\Sigma(x)\,;\,t,\,\Pi_\Sigma(x)) - d_\Sigma(x)\,\}\,\nu_\Sigma(\Pi_\Sigma(x)), \quad 0 \leq t < a,\ x \in \Omega.
\end{equation}
Thus, for fixed $0 \leq t < a$ and $x \in \Sigma$, the mapping $\theta(\,\cdot\,;\,t,\,x): (-\epsilon,\,\epsilon) \longrightarrow (-\epsilon,\,\epsilon)$
is a bijection and we denote its inverse by $\bar{\theta}(\,\cdot\,;\,t,\,x): (-\epsilon,\,\epsilon) \longrightarrow (-\epsilon,\,\epsilon)$.
Note that $\theta(r\,;\,t,\,x) = \bar{\theta}(r\,;\,t,\,x) = r$, if $|r| > 2 \epsilon / 3$ and, therefore,
we may extend $\theta(\,\cdot\,;\,t,\,x)$ and $\bar{\theta}(\,\cdot\,;\,t,\,x)$ as the identity to $\bR \setminus (-\epsilon,\,\epsilon)$
and obtain a pair of inverse bijections.
Since $\Sigma$ is compact, the signed distance extends to a functional $d_\Sigma: \Omega \longrightarrow \bR$ via
\begin{equation*}
	d_\Sigma(x) := \mp \mbox{dist}_\Sigma(x), \qquad x \in \mbox{int}(\Sigma) \cup \Sigma \quad \mbox{resp.} \quad x \in \mbox{ext}(\Sigma) \cup \Sigma,
\end{equation*}
where we denote by $\mbox{ext}(\Sigma) := \Omega \setminus \overline{\mbox{int}}(\Sigma)$ the continuous phase in the transformed setting.
Thus, the spatial inverse $\bar{\Theta}(t,\,\cdot\,): \Omega \longrightarrow \Omega$ of $\Theta(t,\,\cdot\,): \Omega \longrightarrow \Omega$ is given as
\begin{equation}
	\eqnlabel{hanzawa-difference-backward}
	\bar{\Theta}(t,\,x) = x + \{\,\bar{\theta}(d_\Sigma(x)\,;\,t,\,\Pi_\Sigma(x)) - d_\Sigma(x)\,\}\,\nu_\Sigma(\Pi_\Sigma(x)), \quad 0 \leq t < a,\ x \in \Omega.
\end{equation}
\end{subequations}
Note that in both representations \eqnref{hanzawa-difference} the projection onto $\Sigma$ may be interpreted as the multi-valued
metric projection $\Pi_\Sigma: \Omega \longrightarrow 2^\Sigma \setminus \{\,\varnothing\,\}$, since
\begin{equation*}
	\theta(d_\Sigma(x)\,;\,t,\,y) - d_\Sigma(x) = \bar{\theta}(d_\Sigma(x)\,;\,t,\,y) - d_\Sigma(x) = 0,
	\quad \begin{array}{c} 0 \leq t < 0,\ y \in \Sigma, \\[0.5em] x \in \Omega,\ |d_\Sigma(x)| > 2 \epsilon / 3 \end{array}
\end{equation*}
and $\Pi_\Sigma: R(T^\epsilon_\Sigma) \longrightarrow \Sigma$ is well-defined as a single-valued metric projection.

Finally note that the function $\theta(\,\cdot\,;\,t,\,x): \bR \longrightarrow \bR$ is monotonically increasing for all $0 \leq t < a$ and $x \in \Sigma$,
provided that $\|\gamma(t,\,\cdot\,)\|_{BC(\Sigma)} < \epsilon / 4$.
This shows that $\Theta(t,\,\cdot\,): \Omega \longrightarrow \Omega$ is indeed invertible and we have
\begin{equation*}
	\bar{\theta}(r\,|\,t,\,x) = r - \bar{\chi}(r / \epsilon\,;\,t,\,x)\,\gamma(t,\,x), \qquad r \in \bR,\ 0 \leq t < a,\ x \in \Sigma
\end{equation*}
for some function $\bar{\chi}(\,\cdot\,;\,t,\,x): \bR \longrightarrow \bR$ with $\bar{\chi}(r\,;\,t,\,x) = 0$, if $|r| > 2/3$.
Moreover, we have
\begin{equation*}
	\begin{array}{rcrcl}
		\nabla \Theta(t,\,x) & = & 1
			& + & \frac{1}{\epsilon} \chi^\prime(d_\Sigma(x) / \epsilon) \gamma(t,\,\Pi_\Sigma(x)) (1 - P_\Sigma(\Pi_\Sigma(x)))                 \\[0.5em]
			& & & + & \chi(d_\Sigma(x) / \epsilon) \sfN(d_\Sigma(x)) (\nabla_\Sigma \gamma)(t,\,\Pi_\Sigma(x)) \otimes \nu_\Sigma(\Pi_\Sigma(x)) \\[0.5em]
			& & & - & \chi(d_\Sigma(x) / \epsilon) \gamma(t,\,\Pi_\Sigma(x)) \sfN(d_\Sigma(x)) L_\Sigma(\Pi_\Sigma(x))
	\end{array}
\end{equation*}
for $0 \leq t < a,\ x \in \Omega$, which shows that $\nabla \Theta(t,\,\cdot\,)$ is boundedly invertible and $\Theta(t,\,\cdot\,)$ is indeed a $C^1$-diffeomorphism,
provided $\|\gamma(t,\,\cdot\,)\|_{BC(\Sigma)}$ and $\|\nabla_\Sigma\,\gamma(t,\,\cdot\,)\|_{BC(\Sigma)}$ are sufficiently small.
Here $\sfN(\,\cdot\,) = (1 - \cdot\,L_\Sigma)^{-1}$ is defined by \eqnref{normal-transformation}.

\subsection{The Transformed Equations}
For the transformed quantities
\begin{equation*}
	\begin{array}{c}
		\bar{u}(t,\,\cdot\,) := u(t,\,\cdot\,) \circ \Theta(t,\,\cdot\,), \qquad \bar{p}(t,\,\cdot\,) := p(t,\,\cdot\,) \circ \Theta(t,\,\cdot\,) \\[0.5em]
		\bar{c}(t,\,\cdot\,) := c(t,\,\cdot\,) \circ [\Theta(t,\,\cdot\,)]_{\textrm{ext}(\Sigma)}
	\end{array}
\end{equation*}
we obtain transformed differential operators via
\begin{equation*}
	\begin{array}{rcrl}
		\cM^{\bar{u}}(\gamma)\,\bar{\phi}(t,\,x) & := & \{\,D^u_t \phi\,\}\,(t,\,\Theta(t,\,x)),  & \qquad 0 < t < a,\ x \in \Omega \setminus \Sigma, \\[0.5em]
		          \cG(\gamma)\,\bar{\phi}(t,\,x) & := & \{\,\nabla \phi\,\}\,(t,\,\Theta(t,\,x)), & \qquad 0 < t < a,\ x \in \Omega \setminus \Sigma, \\[0.5em]
		          \cL(\gamma)\,\bar{\phi}(t,\,x) & := & \{\,\Delta \phi\,\}\,(t,\,\Theta(t,\,x)), & \qquad 0 < t < a,\ x \in \Omega \setminus \Sigma,
	\end{array}
\end{equation*}
where $\bar{\phi}(t,\,\cdot\,) = \phi(t,\,\cdot\,) \circ \Theta(t,\,\cdot\,): \Omega \setminus \Sigma \longrightarrow \bR$
denotes the transform of a generic scalar valued function $\phi$, as well as
\begin{equation*}
	\cD(\gamma)\,\bar{\Phi}(t,\,x) := \{\,\mbox{div}\,\Phi\,\}\,(t,\,\Theta(t,\,x)), \qquad 0 < t < a,\ x \in \Omega \setminus \Sigma,
\end{equation*}
where $\bar{\Phi}(t,\,\cdot\,) = \Phi(t,\,\cdot\,) \circ \Theta(t,\,\cdot\,): \Omega \setminus \Sigma \longrightarrow \bR^n$
denotes the transform of a generic vector field $\Phi$.
Moreover, we set
\begin{equation*}
	\begin{array}{rcrl}
		\cM^{\bar{u}}(\gamma)\,\bar{\phi}_\Sigma(t,\,x) & := & \{\,D^u_t \phi_\Gamma\,\}\,(t,\,\Theta(t,\,x)),              & \qquad 0 < t < a,\ x \in \Sigma, \\[0.5em]
		   \cG_\Gamma(\gamma)\,\bar{\phi}_\Sigma(t,\,x) & := & \{\,\nabla_{\Gamma(t)} \phi_\Gamma\,\}\,(t,\,\Theta(t,\,x)), & \qquad 0 < t < a,\ x \in \Sigma, \\[0.5em]
		   \cL_\Gamma(\gamma)\,\bar{\phi}_\Sigma(t,\,x) & := & \{\,\Delta_{\Gamma(t)} \phi_\Gamma\,\}\,(t,\,\Theta(t,\,x)), & \qquad 0 < t < a,\ x \in \Sigma,
	\end{array}
\end{equation*}
where $\bar{\phi}_\Sigma(t,\,\cdot\,) = \phi_\Gamma(t,\,\cdot\,) \circ \Theta(t,\,\cdot\,): \Sigma \longrightarrow \bR$
denotes the transform of a generic scalar valued function $\phi_\Gamma$, and, analogously,
\begin{equation*}
	\cD_\Gamma(\gamma)\,\bar{\Phi}_\Sigma(t,\,x) := \{\,\mbox{div}_{\Gamma(t)}\,\Phi_\Gamma\,\}\,(t,\,\Theta(t,\,x)), \qquad 0 < t < a,\ x \in \Sigma,
\end{equation*}
where $\bar{\Phi}_\Sigma(t,\,\cdot\,) = \Phi_\Gamma(t,\,\cdot\,) \circ \Theta(t,\,\cdot\,): \Sigma \longrightarrow \bR^n$
denotes the transform of a generic vector field $\Phi_\Gamma$.
Finally, we transform the geometric quantities
\begin{equation*}
	\begin{array}{rcll}
		   \bar{\nu}_\Gamma(\gamma)\,(t,\,x) & := & \nu_{\Gamma(t)}(t,\,\Theta(t,\,x)),    & \qquad 0 < t < a,\ x \in \Sigma, \\[0.5em]
		\bar{\kappa}_\Gamma(\gamma)\,(t,\,x) & := & \kappa_{\Gamma(t)}(t,\,\Theta(t,\,x)), & \qquad 0 < t < a,\ x \in \Sigma, \\[0.5em]
		     \bar{V}_\Gamma(\gamma)\,(t,\,x) & := & V_\Gamma(t,\,\Theta(t,\,x)),           & \qquad 0 < t < a,\ x \in \Sigma
	\end{array}
\end{equation*}
and arrive at
\begin{equation}
	\eqnlabel{model:fixed}
	\begin{array}{rclll}
		                                                \rho \cM^{\bar{u}}(\gamma) \bar{u} - \eta \cL(\gamma) \bar{u} + \cG(\gamma) \bar{p} & = & \rho \bar{f}                                                       & \ \mbox{in} & (0,\,a) \times \Omega \setminus \Sigma, \\[0.5em]
		                                                                                                                \cD(\gamma) \bar{u} & = & 0                                                                  & \ \mbox{in} & (0,\,a) \times \Omega \setminus \Sigma, \\[0.5em]
		                                                                              \cM^{\bar{u}}(\gamma) \bar{c} - d \cL(\gamma) \bar{c} & = & 0                                                                  & \ \mbox{in} & (0,\,a) \times \mbox{ext}(\Sigma),      \\[0.5em]
		                                                                           [\![\bar{u}]\!]_\Sigma = 0, \quad \bar{V}_\Gamma(\gamma) & = & \bar{u} \cdot \nu_\Gamma(\gamma)                                   & \ \mbox{on} & (0,\,a) \times \Sigma,                  \\[0.5em]
		                           - [\![\eta( \cG(\gamma) \bar{u} + \cG(\gamma) \bar{u}^{\sf{T}}) - \bar{p}]\!]_\Sigma\,\nu_\Gamma(\gamma) \\[0.5em]
		                                                               - \cD_\Gamma(\gamma) \{\,\sigma(\bar{c}_\Sigma) P_\Gamma(\gamma)\,\} & = & 0                                                                  & \ \mbox{on} & (0,\,a) \times \Sigma,                  \\[0.5em]
		                                                                                                           \alpha([\bar{c}]_\Sigma) & = & \bar{c}_\Sigma                                                     & \ \mbox{on} & (0,\,a) \times \Sigma,                  \\[0.5em]
		\cM^{\bar{u}}(\gamma) \bar{c}_\Sigma + \bar{c}_\Sigma \cD_\Gamma(\gamma) \bar{u} - d_\Gamma \cL_\Gamma(\gamma) \bar{c}_\Sigma & = & d \partial^e_\nu \bar{c}                                                 & \ \mbox{on} & (0,\,a) \times \Sigma,                  \\[0.5em]
		                                                                        [\bar{u}]_{\partial \Omega} = 0, \quad \partial_\nu \bar{c} & = & 0                                                                  & \ \mbox{on} & (0,\,a) \times \partial \Omega,         \\[0.5em]
		                         \bar{u}(0) = \bar{u}_0\ \mbox{in}\ \Omega, \quad \gamma(0) = \gamma_0\ \mbox{on}\ \Sigma, \quad \bar{c}(0) & = & \bar{c}_0                                                          & \ \mbox{in} & \mbox{ext}(\Sigma),
	\end{array}
\end{equation}
where $\partial^e_\nu$ on denotes the normal derivative on $\Sigma = \partial \mbox{ext}(\Sigma) \cap \Sigma$,
which should be interpreted as part of the boundary of $\mbox{ext}(\Sigma)$ in this case, i.\,e.~the normal points from $\mbox{ext}(\Sigma)$ into $\mbox{int}(\Sigma)$.
Furthermore, the operator $P_\Gamma(\gamma) = 1 - \nu_\Gamma(\gamma) \otimes \nu_\Gamma(\gamma)$ denotes the transformed projection $P_{\Gamma(\cdot)}$ and
\begin{equation*}
	\bar{c}_\Sigma(t,\,\cdot\,) := c_\Gamma(t,\,\cdot\,) \circ \Theta(t,\,\cdot\,)
\end{equation*}
denotes the transformed surface specific surfactant concentration, being a function $\bar{c}_\Sigma: (0,\,a) \times \Sigma \longrightarrow \bR_+$,
which explains the change in the subscript.

Note that problem \eqnref{model:fixed} may be interpreted as an evolution equation for the transformed unkowns $\bar{u}$ and $\bar{c}$ in the fixed domain
$\Omega \setminus \Sigma$ resp.\ $\mbox{ext}(\Sigma)$.
The pressure $\bar{p}$ plays the role of a Lagrangian parameter introduced due to the divergence constraint.
Moreover, the velocity $\bar{u}$ is subject to a dynamic transmission condition across $\Sigma$,
which introduces the parametrization $\gamma$ and the corresponding evolution equation on $\Sigma$.
Finally, the surfactant concentration $\bar{c}$ is subject to a dynamic boundary condition on $\Sigma$,
which introduces the surface concentration $\bar{c}_\Sigma$ of adsorbed surfactant and the corresponding evolution equation on $\Sigma$.
Indeed, we have
\begin{equation}
	\eqnlabel{normal-transformation}
	\begin{array}{c}
		\nu_\Gamma(\gamma) = \mu(\gamma) (\nu_\Sigma - \sfN(\gamma) \nabla_\Sigma \gamma), \\[0.5em]
		\mu(\gamma) = (1 + |\sfN(\gamma) \nabla_\Sigma \gamma|^2)^{-1/2}, \qquad \sfN(\gamma) = (1 - \gamma L_\Sigma)^{-1},
	\end{array}
\end{equation}
where $L_\Sigma := - \nabla_\Sigma \nu_\Sigma$ denotes the curvature tensor of $\Sigma$, and the kinematic condition reads
\begin{equation}
	\eqnlabel{transformed-kinematic}
	\begin{array}{rcl}
		\partial_t \gamma
			& = & \mu(\gamma)^{-1} \partial_t \gamma\,(\nu_\Sigma \cdot \nu_\Gamma(\gamma)) = \mu(\gamma)^{-1} (\partial_t \Theta \cdot \nu_\Gamma(\gamma)) = \mu(\gamma)^{-1} \bar{u}_\Gamma(\gamma) \\[0.5em]
			& = & \mu(\gamma)^{-1}\,(\bar{u} \cdot \nu_\Gamma(\gamma)) = \bar{u} \cdot \nu_\Sigma - \bar{u} \cdot \sfN(\gamma) \nabla_\Sigma \gamma \qquad \mbox{on}\ (0,\,a) \times \Sigma,
	\end{array}
\end{equation}
which is the hidden evolution equation for $\gamma$.

To keep this section short, we refrain from giving all details of the involved differential geometric computations here.
For a coincise introduction of this topic we refer to \cite{Kimura:Moving-Hypersurfaces},
see also \cite[Section 2]{Koehne-Pruess-Wilke:Two-Phase-Navier-Stokes}.

\subsection{The Linearized Operators}
Our analysis of problem \eqnref{model:fixed} relies on maximal $L_p$-regularity results for suitable linearizations.
To obtain such a linearization w.\,r.\,t.\ reference functions
\begin{equation*}
	\begin{array}{rcrcl}
		       u^\ast & \in & \bX_u(a) & := & \left\{ \begin{array}{l} \Phi \in H^1_p((0,\,a),\,L_p(\Omega \setminus \Sigma)^n) \\[0.5em] \qquad \qquad \cap \ L_p((0,\,a),\,H^2_p(\Omega \setminus \Sigma)^n) \end{array} : \begin{array}{c} [\![\Phi]\!]_\Sigma = 0, \\[0.5em] [\Phi]_{\partial \Omega} = 0 \end{array} \right\}, \\[2.0em]
		       c^\ast & \in & \bX_c(a) & := & \left\{ \begin{array}{l} \phi \in H^1_p((0,\,a),\,L_p(\mbox{ext}(\Sigma))) \\[0.5em] \qquad \qquad \cap \ L_p((0,\,a),\,H^2_p(\mbox{ext}(\Sigma))) \end{array} : \begin{array}{c} [\phi]_\Sigma \in \bX_s(a), \\[0.5em] \partial_\nu \phi = 0 \end{array} \right\},                   \\[2.0em]
		c^\ast_\Sigma & \in & \bX_s(a) & := & H^1_p((0,\,a),\,L_p(\Sigma)) \cap L_p((0,\,a),\,H^2_p(\Sigma))
	\end{array}
\end{equation*}
first observe that
\begin{equation*}
	\begin{array}{c}
		\cM^{\bar{u}}(\gamma) = \partial_t + (\bar{u} - \partial_t \Theta) \cdot \cG(\gamma) = \partial_t + (u^\ast \cdot \nabla) - \sfM(\bar{u},\,\gamma\,|\,u^\ast) \cdot \nabla, \\[0.5em]
		\cG(\gamma) = (1 - \sfG(\gamma)) \nabla, \quad \cD(\gamma) = (1 - \sfG(\gamma)) \nabla\,\cdot = \mbox{div} - \sfD(\gamma) : \nabla, \\[0.5em]
		\cL(\gamma) = \cD(\gamma) \cG(\gamma) = \Delta - \sfA(\gamma) \cdot \nabla - \sfL(\gamma) : \nabla^2,
	\end{array}
\end{equation*}
where the multiplication operators are given as
\begin{equation*}
	\begin{array}{c}
		\sfM(\bar{u},\,\gamma\,|\,u^\ast) = \partial_t \Theta - (\bar{u} - u^\ast) + \sfD(\gamma) (\bar{u} - \partial_t \Theta), \\[0.5em]
		\sfG(\gamma) = (\nabla \Theta)^{-1} (\nabla \Theta - 1), \quad \sfD(\gamma) = \sfG(\gamma)^{\sf{T}},   \\[0.5em]
		\sfA(\gamma) = - (\Delta \bar{\Theta}) \circ \Theta, \quad \sfL(\gamma) = \sfD(\gamma) - \sfD(\gamma) \sfG(\gamma) + \sfG(\gamma).
	\end{array}
\end{equation*}
Unfortunately, the transformation of the operators that act on functions defined on $\mbox{gr}(\Gamma(\,\cdot\,))$ is somehow more involved.
First note that a straight forward computation shows that
\begin{equation*}
	\cG_\Gamma(\gamma) = P_\Gamma(\gamma) \sfN(\gamma) \nabla_\Sigma =: (1 - \sfG_\Sigma(\gamma)) \nabla_\Sigma.
\end{equation*}
Indeed, we have
\begin{equation*}
	\begin{array}{rcl}
		P_\Gamma(\gamma)
			& = & 1 - \nu_\Gamma(\gamma) \otimes \nu_\Gamma(\gamma) = 1 - \mu(\gamma)^2 (\nu_\Sigma - \sfN(\gamma) \nabla_\Sigma \gamma) \otimes (\nu_\Sigma - \sfN(\gamma) \nabla_\Sigma \gamma) \\[0.5em]
			& = & 1 - (1 - \sfn(\gamma))\,(\nu_\Sigma - \sfN(\gamma) \nabla_\Sigma \gamma) \otimes (\nu_\Sigma - \sfN(\gamma) \nabla_\Sigma \gamma) =: P_\Sigma - \sfP(\gamma),
	\end{array}
\end{equation*}
where we have set
\begin{equation*}
	\sfn(\gamma) := \frac{|\sfN(\gamma) \nabla_\Sigma \gamma|^2}{1 + |\sfN(\gamma) \nabla_\Sigma \gamma|^2},
\end{equation*}
and, therefore,
\begin{equation*}
	\sfG_\Sigma(\gamma) = (1 - \sfN(\gamma)) + (1 - P_\Gamma(\gamma)) \sfN(\gamma).
\end{equation*}
To transform the material derivative for functions defined on $\mbox{gr}(\Gamma(\,\cdot\,))$ we assume
\begin{equation*}
	\phi_\Gamma(t,\,\cdot\,): \Gamma(t) \longrightarrow \bR, \qquad 0 < t < a
\end{equation*}
to be a generic family of scalar valued functions.
Setting
\begin{equation*}
	\bar{\phi}_\Sigma(t,\,\cdot\,) := \phi_\Gamma(t,\,\cdot\,) \circ \Theta(t,\,\cdot\,): \Sigma \longrightarrow \bR, \qquad 0 < t < a
\end{equation*}
we obtain for fixed $0 < t < a$ and $x \in \Sigma$ the relation
\begin{equation*}
	\begin{array}{l}
		\cM^{\bar{u}}(\gamma)\,\bar{\phi}_\Sigma(t,\,x) \\[1.0em]
			\ \quad = (D^u_t \phi_\Gamma)(t,\,\Theta(t,\,x)) \\[1.0em]
			\ \quad = {\left. {\displaystyle{\frac{\mbox{d}}{\mbox{d}s}}} \phi_\Gamma(s,\,\chi^u(s\,|\,t,\,\Theta(t,\,x))) \right|}_{s = t} \\[1.5em]
			\ \quad = {\left. {\displaystyle{\frac{\mbox{d}}{\mbox{d}s}}} \phi_\Gamma(s,\,\Theta(s,\,x)) \right|}_{s = t} + {\left. {\displaystyle{\frac{\mbox{d}}{\mbox{d}s}}} \Big( \phi_\Gamma(s,\,\chi^u(s\,|\,t,\,\Theta(t,\,x))) - \phi_\Gamma(s,\,\Theta(s,\,x)) \Big) \right|}_{s = t} \\[1.5em]
			\ \quad = \partial_t \bar{\phi}_\Sigma(t,\,x) + (\bar{u}_\Sigma \cdot \cG_\Gamma(\gamma))\,\bar{\phi}_\Sigma(t,\,x),
	\end{array}
\end{equation*}
where $\bar{u}_\Sigma = P_\Sigma \bar{u}$ denotes the tangential velocity on the reference interface $\Sigma$ and
$\chi^u(\,\cdot\,|\,t,\,x)$ denotes the characteristic curve of a particle, which is advected by the velocity field $u$,
as defined in the introduction, i.\,e.
\begin{equation*}
	\dot{\chi}^u(s\,|\,t,\,x) = u(s,\,\chi^u(s\,|\,t,\,x)), \quad |s - t| < \epsilon, \qquad \chi^u(t\,|\,t,\,x) = x.
\end{equation*}
Indeed, we have
\begin{equation*}
	\begin{array}{rll}
		\chi^u(s\,|\,t,\,\Theta(t,\,x)) - \Theta(t,\,x) & = & (s - t)\,u(s,\,\chi^u(s\,|\,t,\,\Theta(t,\,x))) + o(s - t) \\[0.75em] & =: & (s - t)\,h_1(s\,|\,t,\,x) \\[1.0em]
		\Theta(s,\,x) - \Theta(t,\,x) & = & (s - t)\,{\displaystyle{\frac{\gamma(s,\,x) - \gamma(t,\,x)}{s - t}}} \nu_\Sigma(x) \\[1.0em] & =: & (s - t)\,h_0(s\,|\,t,\,x)
	\end{array}
\end{equation*}
with
\begin{equation*}
	\begin{array}{rcl}
		{\displaystyle{\lim_{s \rightarrow t}}} \Big( h_1(s\,|\,t,\,x) - h_0(s\,|\,t,\,x) \Big)
			& = & \bar{u}(t,\,x) - \partial_t \gamma(t,\,x) \nu_\Sigma(x) \\[0.5em]
			& = & P_\Sigma(x) \bar{u}(t,\,x) = \bar{u}_\Sigma(t,\,x)
	\end{array}
\end{equation*}
due to the kinematic condition \eqnref{transformed-kinematic} and, hence,
\begin{equation*}
	\begin{array}{l}
		{\left. {\displaystyle{\frac{\mbox{d}}{\mbox{d}s}}} \Big( \phi_\Gamma(s,\,\chi^u(s\,|\,t,\,\Theta(t,\,x))) - \phi_\Gamma(s,\,\Theta(s,\,x)) \Big) \right|}_{s = t} \\[1.5em]
			\qquad = {\displaystyle{\lim_{s \rightarrow t}}} \frac{(\phi_\Gamma(s,\,\chi^u(s\,|\,t,\,\Theta(t,\,x))) - \phi_\Gamma(t,\,\Theta(t,\,x))) - (\phi_\Gamma(s,\,\Theta(s,\,x)) - \phi_\Gamma(t,\,\Theta(t,\,x))) }{s - t} \\[1.5em]
			\qquad = {\displaystyle{\lim_{s \rightarrow t}}} \frac{\phi_\Gamma(s,\,\Theta(t,\,x) + (s - t)\,h_1(s,\,t,\,x)) - \phi_\Gamma(s,\,\Theta(t,\,x) + (s - t)\,h_0(s,\,t,\,x))}{s - t} \\[1.5em]
			\qquad = {\displaystyle{\lim_{s \rightarrow t}}} \Big( h_1(s\,|\,t,\,x) - h_0(s\,|\,t,\,x) \Big) \cdot (\nabla_\Gamma \phi_\Gamma)(t,\,\Theta(t,\,x)) \\[1.0em]
			\qquad = (\bar{u}_\Sigma \cdot \cG_\Gamma(\gamma))\,\bar{\phi}_\Sigma(t,\,x)
	\end{array}
\end{equation*}
Hence, the transformed material derivative reads
\begin{equation*}
	\cM^{\bar{u}}(\gamma) = \partial_t + \bar{u}_\Sigma \cdot \cG_\Gamma(\gamma) = \partial_t + (u^\ast_\Sigma \cdot \nabla_\Sigma) - \sfM_\Sigma(\bar{u},\,\gamma\,|\,u^\ast) \cdot \nabla_\Sigma
\end{equation*}
with
\begin{equation*}
	\sfM_\Sigma(\bar{u},\,\gamma\,|\,u^\ast) = \sfD_\Sigma(\gamma) \bar{u}_\Sigma - (\bar{u}_\Sigma - u^\ast_\Sigma), \qquad \sfD_\Sigma(\gamma) = \sfG^\sfT_\Sigma(\gamma).
\end{equation*}
Finally, we need to linearize the curvature
\begin{equation*}
	\kappa_\Gamma(\gamma) = \mu(\gamma)\,\left( \begin{array}{l} \mbox{tr}(\sfN(\gamma)(\nabla_\Sigma \sfN(\gamma) \nabla_\Sigma \gamma) + L_\Sigma) \\[0.5em] \qquad \qquad - \ \mu(\gamma)^2 \sfN(\gamma)^2 \nabla_\Sigma \gamma \cdot (\nabla_\Sigma \sfN(\gamma) \nabla_\Sigma \gamma) \sfN(\gamma) \nabla_\Sigma \gamma \end{array} \right),
\end{equation*}
which becomes
\begin{equation*}
	\kappa^\prime(0) = \mbox{tr}\,L^2_\Sigma + \Delta_\Sigma,
\end{equation*}
cf.~\cite{Kimura:Moving-Hypersurfaces} and \cite[Section 2]{Koehne-Pruess-Wilke:Two-Phase-Navier-Stokes}.

\subsection{The Principal Linear Part}
Based on the above preparations the principal linearization of problem \eqnref{model:fixed} reads as follows.
Note that here we drop the bars that have been introduced in the previous subsections for convenience.
\begin{equation}
	\eqnlabel{model:linearized}
	\begin{array}{rclll}
		                                                                             \rho \partial_t u + (u^\ast \cdot \nabla) u - \eta \Delta u + \nabla p & = & \rho F_u(u,\,p,\,\gamma\,|\,u^\ast)                                  & \mbox{in} & (0,\,a) \times \Omega \setminus \Sigma, \\[0.5em]
		                                                                                                                                      \mbox{div}\,u & = & G(u)                                                                 & \mbox{in} & (0,\,a) \times \Omega \setminus \Sigma, \\[0.5em]
		                                                                                                \partial_t c + (u^\ast \cdot \nabla) c - d \Delta c & = & F_c(u,\,\gamma,\,c\,|\,u^\ast)                                       & \mbox{in} & (0,\,a) \times \mbox{ext}(\Sigma),      \\[0.5em]
	                                                                                                                                     [\![u]\!]_\Sigma & = & 0                                                                    & \mbox{on} & (0,\,a) \times \Sigma,                  \\[0.5em]
		                                                                          - P_\Sigma [\![\eta( \nabla u + \nabla u^{\sf{T}})]\!]_\Sigma\,\nu_\Sigma                                                                                                                                  \\[0.5em]
		                                                                                            - \ \sigma^\prime(c^\ast_\Sigma) \nabla_\Sigma c_\Sigma & = & H^\tau_u(u,\,[\![p]\!]_\Sigma,\,\gamma,\,c_\Sigma\,|\,c^\ast_\Sigma) & \mbox{on} & (0,\,a) \times \Sigma,                  \\[0.5em]
		                                                                  - [\![\eta( \nabla u + \nabla u^{\sf{T}})]\!]_\Sigma\,\nu_\Sigma \cdot \nu_\Sigma                                                                                                                                   \\[0.5em]
		                                                                                + \ [\![p]\!]_\Sigma - \sigma(c^\ast_\Sigma) \Delta_\Sigma c_\Sigma & = & H^\nu_u(u,\,\gamma,\,c_\Sigma\,|\,c^\ast_\Sigma)                     & \mbox{on} & (0,\,a) \times \Sigma,                  \\[0.5em]
		                                                                \partial_t \gamma - u \cdot \nu_\Sigma + (u^\ast_\Sigma \cdot \nabla_\Sigma) \gamma & = & H_\gamma(u,\,\gamma\,|\,u^\ast)                                      & \mbox{on} & (0,\,a) \times \Sigma,                  \\[0.5em]
		                                                                                               \alpha^\prime([c^\ast]_\Sigma) [c]_\Sigma - c_\Sigma & = & H^\alpha_c(c\,|\,c^\ast)                                             & \mbox{on} & (0,\,a) \times \Sigma,                  \\[0.5em]
		                                               \partial_t c_\Sigma + (u^\ast_\Sigma \cdot \nabla_\Sigma) c_\Sigma - d_\Gamma \Delta_\Sigma c_\Sigma & = & H^\epsilon_c(c,\,c_\Sigma\,|\,u^\ast)                                & \mbox{on} & (0,\,a) \times \Sigma,                  \\[0.5em]
		                                                                                                    [u]_{\partial \Omega} = 0, \quad \partial_\nu c & = & 0                                                                    & \mbox{on} & (0,\,a) \times \partial \Omega,         \\[0.5em]
		                                                                                                     u(0) = u_0\ \mbox{in}\ \Omega, \quad \gamma(0) & = & \gamma_0                                                             & \mbox{on} & \Sigma,                                 \\[0.5em]
		                                                                                                                                               c(0) & = & c_0                                                                  & \mbox{in} & \mbox{ext}(\Sigma),
	\end{array}
\end{equation}
where the non-linearities are given as
\begin{equation*}
	\begin{array}{rcl}
		                          \rho F_u(u,\,p,\,\gamma\,|\,u^\ast) & = & \{\,\rho \sfM(u,\,\gamma\,|\,u^\ast) \cdot \nabla - \eta \sfA(\gamma) \cdot \nabla - \eta \sfL(\gamma) : \nabla^2\,\}\,u \\[0.5em]
		                                                              &   & \ + \ \{\,\sfG(\gamma) \nabla\,\}\,p, \\[0.5em]
		                                                         G(u) & = & \{\,\sfD(\gamma) : \nabla\,\}\,u, \\[0.5em]
		                               F_c(u,\,\gamma,\,c\,|\,u^\ast) & = & \{\,\sfM(u,\,\gamma\,|\,u^\ast) \cdot \nabla + d \sfA(\gamma) \cdot \nabla + d \sfL(\gamma) : \nabla^2\,\}\,c \\[0.5em]
	\end{array}
\end{equation*}
as well as
\begin{equation*}
	\begin{array}{rcl}
		H^\tau_u(u,\,[\![p]\!]_\Sigma,\,\gamma,\,c_\Sigma\,|\,u^\ast) & = & P_\Sigma [\![\eta(\sfG(\gamma) \nabla u + \sfG(\gamma) \nabla u^{\sf{T}})]\!]_\Sigma\,\sfN(\gamma) \nabla_\Sigma \gamma \\[0.5em]
		                                                              &   & \ - \ P_\Sigma [\![\eta(\nabla u + \nabla u^{\sf{T}})]\!]_\Sigma\,\sfN(\gamma) \nabla_\Sigma \gamma \\[0.5em]
		                                                              &   & \ + \ [\![p]\!]_\Sigma\,\sfN(\gamma) \nabla_\Sigma \gamma - P_\Sigma [\![\eta(\sfG(\gamma) \nabla u + \sfG(\gamma) \nabla u^{\sf{T}})]\!]_\Sigma\,\nu_\Sigma \\[0.5em]
		                                                              &   & \ + \ \mu(\gamma)^{-1}\,(1 - \mu(\gamma)) \sigma^\prime(c^\ast_\Sigma) \nabla_\Sigma c_\Sigma \\[0.5em]
		                                                              &   & \ + \ \mu(\gamma)^{-1}\,(\sigma^\prime(c_\Sigma) - \sigma^\prime(c^\ast_\Sigma)) \nabla_\Sigma c_\Sigma \\[0.5em]
		                                                              &   & \ - \ \mu(\gamma)^{-1}\,\sigma^\prime(c_\Sigma) \sfG_\Sigma(\gamma) \nabla_\Sigma c_\Sigma \\[0.5em]
		                                                              &   & \ - \ \sigma(c_\Sigma) \kappa_\Gamma(\gamma) \sfN(\gamma) \nabla_\Sigma \gamma, \\[0.5em]
		                    H^\nu_u(u,\,\gamma,\,c_\Sigma\,|\,u^\ast) & = & [\![\eta(\sfG(\gamma) \nabla u + \sfG(\gamma) \nabla u^{\sf{T}})]\!]_\Sigma\,\sfN(\gamma) \nabla_\Sigma \gamma \cdot \nu_\Sigma \\[0.5em]
		                                                              &   & \ - \ [\![\eta(\nabla u + \nabla u^{\sf{T}})]\!]_\Sigma\,\sfN(\gamma) \nabla_\Sigma \gamma \cdot \nu_\Sigma \\[0.5em]
		                                                              &   & \ - \ [\![\eta(\sfG(\gamma) \nabla u + \sfG(\gamma) \nabla u^{\sf{T}})]\!]_\Sigma\,\nu_\Sigma \cdot \nu_\Sigma \\[0.5em]
		                                                              &   & \ + \ \mu(\gamma)^{-1}\,\sigma(c^\ast_\Sigma) (\kappa_\Gamma(\gamma) - \Delta_\Sigma \gamma) \nu_\Gamma(\gamma) \cdot \nu_\Sigma \\[0.5em]
		                                                              &   & \ - \ \mu(\gamma)^{-1}\,\sigma^\prime(c_\Sigma) \sfG_\Sigma(\gamma) \nabla_\Sigma c_\Sigma \cdot \nu_\Sigma \\[0.5em]
		                                                              &   & \ + \ \mu(\gamma)^{-1}\,(\sigma(c_\Sigma) - \sigma(c^\ast_\Sigma)) \kappa_\Gamma(\gamma) \nu_\Gamma(\gamma) \cdot \nu_\Sigma, \\[0.5em] 
		                              H_\gamma(u,\,\gamma\,|\,u^\ast) & = & \{\,(1 - \sfN(\gamma)) u_\Sigma \cdot \nabla_\Sigma - (u_\Sigma - u^\ast_\Sigma) \cdot \nabla_\Sigma\,\}\,\gamma, \\[0.5em]
		                                     H^\alpha_c(c\,|\,c^\ast) & = & \alpha^\prime([c^\ast]_\Sigma) [c]_\Sigma - \alpha([c]_\Sigma), \\[0.5em]
		                        H^\epsilon_c(c,\,c_\Sigma\,|\,u^\ast) & = & d \partial^e_\nu \bar{c} - c_\Sigma\,\cD_\Gamma(\gamma) u \\[0.5em]
		                                                              &   & \ + \ \{\,\sfM_\Sigma(u,\,\gamma\,|\,u^\ast) \cdot \nabla_\Sigma - d \sfA_\Sigma(\gamma) \cdot \nabla_\Sigma - d \sfL_\Sigma(\gamma)\,\}\,c_\Sigma.
	\end{array}
\end{equation*}
As mentioned above, our analysis of problem \eqnref{model:fixed} -- or equivalently problem \eqnref{model:linearized} -- relies
on maximal $L_p$-regularity results for suitable linearizations,
which here are given by the linear left-hand side of problem \eqnref{model:linearized}.
These shall be established now.
To be precise, we will prove the existence of a unique maximal regular solution
\begin{equation*}
	(u,\,p,\,\gamma,\,c,\,c_\Sigma) \in \bX(a) := \bX_u(a) \times \bX_p(a) \times \bX_\gamma(a) \times \bX_c(a) \times \bX_s(a)
\end{equation*}
with
\begin{equation*}
	\begin{array}{rcl}
		\bX_p(a)      & := & \left\{\,q \in L_p((0,\,a),\,\dot{H}^1_p(\Omega))\,:\,[\![q]\!]_\Sigma \in \bH^\nu_u(a)\,\right\}, \\[1.0em]
		\bX_\gamma(a) & := & W^{2 - 1/2p}_p((0,\,a),\,L_p(\Sigma)) \cap H^1_p((0,\,a),\,W^{2 - 1/p}_p(\Sigma))                  \\[0.5em]
		              &    & \qquad \qquad \cap \ L_p((0,\,a),\,W^{3 - 1/p}_p(\Sigma))
	\end{array}
\end{equation*}
and $\bX_u(a)$, $\bX_c(a)$ and $\bX_s(a)$ defined as above
to the linear problem
\begin{equation}
	\eqnlabel{model:linear}
	\begin{array}{rclll}
		                                                                           \rho \partial_t u + (u^\ast \cdot \nabla) u - \eta \Delta u + \nabla p & = & \rho f_u             & \mbox{in} & (0,\,a) \times \Omega \setminus \Sigma, \\[0.5em]
		                                                                                                                                    \mbox{div}\,u & = & g                    & \mbox{in} & (0,\,a) \times \Omega \setminus \Sigma, \\[0.5em]
		                                                                                              \partial_t c + (u^\ast \cdot \nabla) c - d \Delta c & = & f_c                  & \mbox{in} & (0,\,a) \times \mbox{ext}(\Sigma),      \\[0.5em]
	                                                                                                                                   [\![u]\!]_\Sigma & = & 0                    & \mbox{on} & (0,\,a) \times \Sigma,                  \\[0.5em]
		                  - P_\Sigma [\![\eta( \nabla u + \nabla u^{\sf{T}})]\!]_\Sigma\,\nu_\Sigma - \sigma^\prime(c^\ast_\Sigma) \nabla_\Sigma c_\Sigma & = & h^\tau_u             & \mbox{on} & (0,\,a) \times \Sigma,                  \\[0.5em]
		- [\![\eta( \nabla u + \nabla u^{\sf{T}})]\!]_\Sigma\,\nu_\Sigma \cdot \nu_\Sigma + [\![p]\!]_\Sigma - \sigma(c^\ast_\Sigma) \Delta_\Sigma \gamma & = & h^\nu_u              & \mbox{on} & (0,\,a) \times \Sigma,                  \\[0.5em]
		                                                              \partial_t \gamma - u \cdot \nu_\Sigma + (u^\ast_\Sigma \cdot \nabla_\Sigma) \gamma & = & h_\gamma             & \mbox{on} & (0,\,a) \times \Sigma,                  \\[0.5em]
		                                                                                             \alpha^\prime([c^\ast]_\Sigma) [c]_\Sigma - c_\Sigma & = & h^\alpha_c           & \mbox{on} & (0,\,a) \times \Sigma,                  \\[0.5em]
		                                             \partial_t c_\Sigma + (u^\ast_\Sigma \cdot \nabla_\Sigma) c_\Sigma - d_\Gamma \Delta_\Sigma c_\Sigma & = & h^\epsilon_c         & \mbox{on} & (0,\,a) \times \Sigma,                  \\[0.5em]
		                                                                                                  [u]_{\partial \Omega} = 0, \quad \partial_\nu c & = & 0                    & \mbox{on} & (0,\,a) \times \partial \Omega,         \\[0.5em]
		                                                         u(0) = u_0\ \mbox{in}\ \Omega, \quad \gamma(0) = \gamma_0\ \mbox{on}\ \Sigma, \quad c(0) & = & c_0                  & \mbox{in} & \mbox{ext}(\Sigma),
	\end{array}
\end{equation}
provided the data satisfy the necessary regularity and compatibility conditions.
Here the necessary regularity conditions are given as
\begin{equation*}
	\begin{array}{rclcl}
		         f_u & \in & \bF_u(a)          & := & L_p((0,\,a),\,L_p(\Omega)^n), \\[0.5em]
		           g & \in & \bG(a)            & := & H^1_p((0,\,a),\,{}_0 \dot{H}^{-1}_p(\Omega)) \cap H^{1/2}_p((0,\,a),\,L_p(\Omega)) \\[0.5em]
		             &     &                   &    & \qquad \qquad \cap \ L_p((0,\,a),\,H^1_p(\Omega \setminus \Sigma)), \\[0.5em]
		         f_c & \in & \bF_c(a)          & := & L_p((0,\,a),\,L_p(\mbox{ext}(\Sigma))), \\[0.5em]
		    h^\tau_u & \in & \bH^\tau_u(a)     & := & W^{1/2 - 1/2p}_p((0,\,a),\,L_p(\Sigma,\,T\Sigma)) \cap L_p((0,\,a),\,W^{1 - 1/p}_p(\Sigma,\,T\Sigma)), \\[0.5em]
		     h^\nu_u & \in & \bH^\nu_u(a)      & := & W^{1/2 - 1/2p}_p((0,\,a),\,L_p(\Sigma)) \cap L_p((0,\,a),\,W^{1 - 1/p}_p(\Sigma)), \\[0.5em]
		    h_\gamma & \in & \bH_\gamma(a)     & := & W^{1 - 1/2p}_p((0,\,a),\,L_p(\Sigma)) \cap L_p((0,\,a),\,W^{2 - 1/p}_p(\Sigma)), \\[0.5em]
		  h^\alpha_c & \in & \bH^\alpha_c(a)   & := & \bX_s(a), \\[0.5em]
		h^\epsilon_c & \in & \bH^\epsilon_c(a) & := & L_p((0,\,a),\,L_p(\Sigma)), \\[0.5em]
		         u_0 & \in & \bT_u             & := & \left\{\,\Phi \in W^{2 - 2/p}_p(\Omega \setminus \Sigma)^n\,:\,[\![\Phi]\!]_\Sigma = 0,\ [\Phi]_{\partial \Omega} = 0\,\right\}, \\[0.5em]
		    \gamma_0 & \in & \bT_\gamma        & := & W^{3 - 2/p}_p(\Sigma), \\[0.5em]
		         c_0 & \in & \bT_c             & := & \left\{\,\phi \in W^{2 - 2/p}_p(\mbox{ext}(\Sigma))\,:\,[\phi]_\Sigma \in W^{2 - 2/p}_p(\Sigma),\ \partial_\nu \phi = 0\,\right\}.
	\end{array}
\end{equation*}
Note that all these spaces arise based on trace theorems for the solution space $\bX(a)$.
The first regularity assertion for the right-hand side $g$ of the divergence constraint is a result of partial integration.
Here ${}_0 \dot{H}^{-1}_p(\Omega) := \dot{H}^1_{p^\prime}(\Omega)^\prime$ with $1/p + 1/p^\prime = 1$.
With the abbreviations
\begin{equation*}
	\begin{array}{rcl}
		\bY(a) & := & \bF_u(a) \times \bG(a) \times \bF_c(a) \times \bH^\tau_u(a) \times \bH^\nu_u(a) \times \bH_\gamma(a) \times \bH^\alpha_c(a) \times \bH^\epsilon_c(a), \\[0.5em]
		   \bT & := & \bT_u \times \bT_\gamma \times \bT_c
	\end{array}
\end{equation*}
our main result concerning \eqnref{model:linear} now reads as follows.
\begin{theorem}
	\thmlabel{maximal-regularity}
	Let $a > 0$, let $\Omega \subseteq \bR^n$ be a bounded domain with boundary of class $C^{3-}$, let $\Sigma \in \cM\cH^2(\Omega)$ be of class $C^{3-}$ and let $p > n + 2$.
	Let $\rho_\pm,\,\eta_\pm,\,d,\,d_\Gamma > 0$ and let $\sigma,\,\alpha \in C^{3-}(\bR_+,\,\bR_+)$ with $\alpha^\prime > 0$.
	Moreover, let
	\begin{equation*}
		u^\ast \in \bX_u(a), \quad c^\ast \in \bX_c(a), \quad c^\ast_\Sigma \in \bX_s(a).
	\end{equation*}
	Then the linear problem \eqnref{model:linear} admits a unique maximal regular solution
	\begin{equation*}
		(u,\,p,\,\gamma,\,c,\,c_\Sigma) \in \bX(a),
	\end{equation*}
	if and only if the data satisfy the regularity conditions
	\begin{equation*}
		(f_u,\,g,\,f_c,\,h^\tau_u,\,h^\nu_u,\,h_\gamma,\,h^\alpha_c,\,h^\epsilon_c) \in \bY(a),
		\qquad (u_0,\,\gamma_0,\,c_0) \in \bT
	\end{equation*}
	and the compatibility conditions
	\begin{equation*}
		\mbox{\upshape div}\,u_0 = g(0), \qquad - P_\Sigma [\![\eta( \nabla u_0 + \nabla u^{\sf{T}}_0)]\!]_\Sigma\,\nu_\Sigma - \sigma^\prime(c^\ast_\Sigma(0)) \nabla_\Sigma s_0 = h^\tau_u(0)
	\end{equation*}
	with
	\begin{equation*}
		s_0 = \alpha^\prime([c^\ast(0)]_\Sigma) [c_0]_\Sigma - h^\alpha_c(0).
	\end{equation*}
	Moreover, the solutions depend continuously on the data.
\end{theorem}
\begin{proof}
	We start by determining $c_\Sigma \in \bX_s(a)$ as the unique maximal regular solution to the parabolic problem
	\begin{equation*}
		\begin{array}{rclll}
			\partial_t c_\Sigma + (u^\ast_\Sigma \cdot \nabla_\Sigma) c_\Sigma - d_\Gamma \Delta_\Sigma c_\Sigma & = & h^\epsilon_c & \mbox{on} & (0,\,a) \times \Sigma, \\[0.5em]
			                                                                                         c_\Sigma(0) & = & s_0          & \mbox{on} & \Sigma,
		\end{array}
	\end{equation*}
	which then enables us to obtain $c \in \bX_c(a)$ as the unique maximal regular solution to the parabolic problem
	\begin{equation*}
		\begin{array}{rclll}
			\partial_t c + (u^\ast \cdot \nabla) c - d \Delta c & = & f_c                   & \mbox{in} & (0,\,a) \times \mbox{ext}(\Sigma), \\[0.5em]
		                  \alpha^\prime([c^\ast]_\Sigma) [c]_\Sigma & = & c_\Sigma + h^\alpha_c & \mbox{on} & (0,\,a) \times \Sigma,             \\[0.5em]
			                                           \partial_\nu c & = & 0                     & \mbox{on} & (0,\,a) \times \partial \Omega,    \\[0.5em]
			                                                     c(0) & = & c_0                   & \mbox{in} & \mbox{ext}(\Sigma).
		\end{array}
	\end{equation*}
	Note that at this point we require $\alpha^\prime > 0$ in order to force the boundary condition to satisfy the Lopatinskii-Shapiro condition.
	Now, the remaining system for the unknowns
	\begin{equation*}
		(u,\,p,\,\gamma) \in \bX_u(a) \times \bX_p(a) \times \bX_\gamma(a)
	\end{equation*}
	forms a two-phase Navier-Stokes problem with time and space dependent surface tension $\sigma(c^\ast_\Sigma)$
	and we may invoke \cite[Theorem~3.1 \& Corollary~3.4]{Koehne-Pruess-Wilke:Two-Phase-Navier-Stokes} to complete the proof.
\end{proof}

\section{Local Well-Posedness}
	\seclabel{well-posedness}
\subsection{The Non-Linear Problem}
With the maximal $L_p$-regularity result, \Thmref{maximal-regularity}, at hand
we require only one more ingredient to prove the well-posedness assertion of \Thmref{well-posedness}.
Of course, we also need suitable estimates for the non-linear right-hand sides in \eqnref{model:linearized}.
To shorten the notation we denote by
\begin{equation*}
	L(\,\cdot\,|\,u^\ast,\,c^\ast,\,c^\ast_\Sigma): \bX(a) \longrightarrow \bY(a), \quad \mbox{resp.} \quad N(\,\cdot\,|\,u^\ast,\,c^\ast,\,c^\ast_\Sigma): \bX(a) \longrightarrow \bY(a)
\end{equation*}
the linear operator induced by the left-hand side resp.\ the non-linear operator induced by the right-hand side of \eqnref{model:linearized}
without the inital condition.
Thus, the problem \eqnref{model:linearized} is equivalent to
\begin{equation*}
	\begin{array}{rcl}
		L(u,\,p,\,\gamma,\,c,\,c_\Sigma\,|\,u^\ast,\,c^\ast,\,c^\ast_\Sigma) & = & N(u,\,p,\,\gamma,\,c,\,c_\Sigma\,|\,u^\ast,\,c^\ast,\,c^\ast_\Sigma), \\[0.5em]
		                                           (u(0),\,\gamma(0),\,c(0)) & = & (u_0,\,\gamma_0,\,c_0)
	\end{array}
\end{equation*}
and the operator $L(\,\cdot\,|\,u^\ast,\,c^\ast,\,c^\ast_\Sigma)$ is bounded, since the derived regularity conditions on the data are necessary.
Moreover, the operator
\begin{equation*}
	{}_0 L(\,\cdot\,|\,u^\ast,\,c^\ast,\,c^\ast_\Sigma): {}_0 \bX(a) \longrightarrow {}_0 \bY(a),
\end{equation*}
which denotes the restriction of the operator $L(\,\cdot\,|\,u^\ast,\,c^\ast,\,c^\ast_\Sigma)$ to the closed subspace ${}_0 \bX(a)$ of $\bX(a)$ with vanishing initial values,
is an isomorphism thanks to \Thmref{maximal-regularity}.
Here we denote by ${}_0 \bY(a)$ the closed subspace of $\bY(a)$ with vanishing initial values
and note that the corresponding compatibility conditions on the data are necessarily satisfied in the setting of vanishing initial values.
Now, the non-linear operator $N(\,\cdot\,|\,u^\ast,\,c^\ast,\,c^\ast_\Sigma)$ has the following mapping properties
as has been shown in \cite[Proposition~4.1]{Pruess-Simonett:Two-Phase-Navier-Stokes-Analytic} for the two-phase Navier-Stokes
equations with surface tension in the case where $\Sigma = \bR^{n - 1}$ and in \cite[Proposition~4.2]{Koehne-Pruess-Wilke:Two-Phase-Navier-Stokes}
in the case of a general geometry as considered here.
\begin{proposition}
	\proplabel{non-linearity}
	Let $a > 0$, let $\Omega \subseteq \bR^n$ be a bounded domain with boundary of class $C^{3-}$, let $\Sigma \in \cM\cH^2(\Omega)$ be of class $C^{3-}$ and let $p > n + 2$.
	Let $\rho_\pm,\,\eta_\pm,\,d,\,d_\Gamma > 0$ and let $\sigma,\,\alpha \in C^{3-}(\bR_+,\,\bR_+)$ with $\alpha^\prime > 0$.
	Moreover, let
	\begin{equation*}
		u^\ast \in \bX_u(a), \quad c^\ast \in \bX_c(a), \quad c^\ast_\Sigma \in \bX_s(a).
	\end{equation*}
	Then
	\begin{equation*}
		N(\,\cdot\,|\,u^\ast,\,c^\ast,\,c^\ast_\Sigma) \in C^1(\bX(a),\,\bY(a))
	\end{equation*}
	and the Fr{\'e}chet derivative $DN(u,\,p,\,\gamma,\,c,\,c_\Sigma\,|\,u^\ast,\,c^\ast,\,c^\ast_\Sigma)$ satisfies the estimate
	\begin{equation*}
		\begin{array}{l}
			M^{-1} {\|DN(u,\,p,\,\gamma,\,c,\,c_\Sigma\,|\,u^\ast,\,c^\ast,\,c^\ast_\Sigma)\|}_{\cB({}_0 \bX(\bar{a}),\,{}_0 \bY(\bar{a}))} \\[0.5em]
				\quad  \leq \ {\|u - u^\ast\|}_{BC((0,\,\bar{a}),\,BC^1)} + {\|c - c^\ast\|}_{BC((0,\,\bar{a}),\,BC^1)} \\[0.5em]
				\qquad    + \ {\|c_\Sigma - c^\ast_\Sigma\|}_{BC((0,\,\bar{a}),\,BC^1)} + {\|(u,\,p,\,\gamma,\,c,\,c_\Sigma)\|}_{\bX(\bar{a})} \\[0.5em]
				\qquad    + \ \big( {\|\gamma\|}_{\bX_\gamma(\bar{a})} + {\|\nabla_\Sigma \gamma\|}_{BC((0,\,\bar{a}),\,BC^1)} \big) {\|(u,\,p,\,\gamma,\,c,\,c_\Sigma)\|}_{\bX(\bar{a})} \\[0.5em]
				\qquad    + \ P\big( {\|\gamma\|}_{\bX_\gamma(\bar{a})},\,{\|\nabla_\Sigma \gamma\|}_{BC((0,\,\bar{a}),\,BC^1)} \big) {\|\gamma\|}_{\bX_\gamma(\bar{a})} \\[0.5em]
				\qquad    + \ Q\big( {\|\nabla_\Sigma \gamma\|}_{BC((0,\,\bar{a}),\,BC^1)} \big) {\|\nabla_\Sigma \gamma\|}_{BC((0,\,\bar{a}),\,BC^1)}
		\end{array}
	\end{equation*}
	for all $\bar{a} \in (0,\,a]$ and all $(u,\,p,\,\gamma,\,c,\,c_\Sigma) \in \bX(\bar{a})$ with some constant $M = M(a) > 0$.
	Here $P$ and $Q$ denote fixed polynomials.
	Moreover, if $\alpha,\,\sigma \in C^\omega(\bR_+,\,\bR_+)$, then
	\begin{equation*}
		N(\,\cdot\,|\,u^\ast,\,c^\ast,\,c^\ast_\Sigma) \in C^\omega(\bX(a),\,\bY(a)).
	\end{equation*}
\end{proposition}
The proofs given in the above mentioned sources carry over to the case considered here.
The main arguments are the polynomial structure of $N(\,\cdot\,|\,u^\ast,\,c^\ast,\,c^\ast_\Sigma)$,
which is equal to that considered in \cite{Koehne-Pruess-Wilke:Two-Phase-Navier-Stokes}, and the embeddings
\begin{equation*}
	\bX_u(\bar{a}),\,\bX_\gamma(\bar{a}),\,\bX_c(\bar{a}) \hookrightarrow BUC((0,\,\bar{a}),\,BUC^1),
\end{equation*}
which are available thanks to $p > n + 2$ and which imply the involved function spaces to be Banach algebras.
With this preparations at hand the proof of \Thmref{well-posedness} may be carried out as follows.

\subsection{Proof of \Thmref{well-posedness}, Step 1}
Suppose the assumptions of \Thmref{well-posedness} are satisfied and
\begin{equation*}
	u_0 \in W^{2 - 2/p}_p(\Omega \setminus \Gamma_0), \quad \Gamma_0 \in W^{3 - 2/p}_p, \quad c_0 \in W^{2 - 2/p}_p(\Omega_+(0), \bR_+)
\end{equation*}
that satisfy the stated regularity and compatibility conditions are given.
According to the considerations in \Secref{linearization} we may approximate $\Gamma_0$ for every given $\delta > 0$
with a real analytic hypersurface $\Sigma \in \cM\cH^2(\Omega)$ such that
\begin{equation*}
	\mbox{dist}_{\cM\cH^2(\Omega)}(\Sigma,\,\Gamma_0) < \delta
\end{equation*}
and $\Gamma_0$ is parametrized over $\Sigma$ by a function $\gamma_0 \in W^{3 - 2/p}_p(\Sigma)$.
Therefore, it is sufficient to consider the transformed problem \eqnref{model:linearized},
where we drop the bars again for convenience.

Now, we denote by $\bar{u} \in \bX_u(a)$ an arbitrary extension of $u_0$,
which may e.\,g.\ be obtained as a solution to a suitable parabolic problem.
Moreover, we denote by $\bar{h}^\alpha_c \in \bH^\alpha_c(a)$ an arbitrary extension of $\alpha^\prime(0) [c_0]_\Sigma$,
which may also be obtained as a solution to a suitable parabolic problem on $\Sigma$.
Finally, we denote by $\bar{g} \in \bG(a)$ resp.\ $\bar{h}^\tau_u \in \bH^\tau_u(a)$ extentions of
\begin{equation*}
	\mbox{div}\,u_0 \qquad \mbox{resp.} \qquad - P_\Sigma [\![\eta(\nabla u_0 + \nabla^{\sf{T}}_0)]\!]_\Sigma\,\nu_\Sigma,
\end{equation*}
which exist thanks to \cite[Proposition~4.1]{Koehne-Pruess-Wilke:Two-Phase-Navier-Stokes},
and solve
\begin{equation*}
	\begin{array}{c}
		(u^\ast,\,p^\ast,\,\gamma^\ast,\,c^\ast,\,c^\ast_\Sigma) \in \bX(a), \quad (u^\ast(0),\,\gamma^\ast(0),\,c^\ast(0)) = (u_0,\,\gamma_0,\,c_0) \\[0.5em]
		\quad L(u^\ast,\,p^\ast,\,\gamma^\ast,\,c^\ast,\,c^\ast_\Sigma\,|\,\bar{u},\,0,\,0) = (f,\,\bar{g},\,0,\,\bar{h}^\tau_u,\,0,\,0,\,\bar{h}^\alpha_c,\,0)
	\end{array}
\end{equation*}
to obtain a reference solution $(u^\ast,\,c^\ast,\,c^\ast_\Sigma) \in \bX_u(a) \times \bX_c(a) \times \bX_s(a)$.

\subsection{Proof of \Thmref{well-posedness}, Step 2}
To prepare a fixed point argument we split the desired solution $z = (u,\,p,\,\gamma,\,c,\,c_\Sigma) \in \bX(\bar{a})$ as
\begin{equation*}
	z = z^\ast + \bar{z}, \qquad \bar{z} \in {}_0 \bX(\bar{a}),
\end{equation*}
where $z^\ast = (u^\ast,\,p^\ast,\,\gamma^\ast,\,c^\ast,\,c^\ast_\Sigma) \in \bX(a)$ denotes the reference solution obtained in the first step
and $\bar{a} \in (0,\,a]$ will be chosen below.
Problem \eqnref{model:linearized} may then be rewritten as
\begin{equation*}
	{}_0 L(\bar{z}\,|\,u^\ast,\,c^\ast,\,c^\ast_\Sigma) = N(\bar{z} + z^\ast\,|\,u^\ast,\,c^\ast,\,c^\ast_\Sigma) - L(z^\ast\,|\,u^\ast,\,c^\ast,\,c^\ast_\Sigma) =: K(\bar{z})
\end{equation*}
and, hence, the solution is given as the fixed point $\bar{z} = {}_0 L(\,\cdot\,|\,u^\ast,\,c^\ast,\,c^\ast_\Sigma)^{-1} K(\bar{z})$.
Now,
\begin{equation*}
	{\|{}_0 L(\,\cdot\,|\,u^\ast,\,c^\ast,\,c^\ast_\Sigma)^{-1}\|}_{\cB({}_0 \bY(\bar{a}),\,{}_0 \bX(\bar{a}))} \leq C, \qquad \bar{a} \in (0,\,a]
\end{equation*}
and due to $K(0) = N(z^\ast\,|\,u^\ast,\,c^\ast,\,c^\ast_\Sigma) - L(z^\ast\,|\,u^\ast,\,c^\ast,\,c^\ast_\Sigma)$ and \Propref{non-linearity}
we may choose $\bar{a} \in (0,\,a]$ and $r > 0$ sufficiently small to ensure
\begin{equation*}
	{\|K(0)\|}_{{}_0 \bY(\bar{a})} < \frac{r}{2 C}, \quad {\|DK(\bar{z})\|}_{\cB({}_0 \bX(\bar{a}),\,{}_0 \bY(\bar{a}))} < \frac{1}{2 C}, \quad \bar{z} \in {}_0 \bX(\bar{a}),\ {\|\bar{z}\|}_{{}_0 \bX(\bar{a})} \leq r.
\end{equation*}
Thus, the contraction mapping principle applies and yields the unique solution $\bar{z} \in {}_0 \bX(\bar{a})$.

\section{Analyticity}
	\seclabel{analyticity}
\subsection{Proof of \Thmref{well-posedness}, Step 3}
In order to complete the proof of \Thmref{well-posedness} it remains to show the analyticity of the solutions.
However, this may be obtained by the well-known parameter-trick since the non-linear right hand side $N(\,\cdot\,|\,u^\ast,\,c^\ast,\,c^\ast_\Sigma)$
is real analytic, provided $\alpha,\,\sigma \in C^\omega(\bR_+,\,\bR_+)$, cf.~\Propref{non-linearity}.
We refrain from giving all the details here, since the proof is similar as that of \cite[Theorem~4.3]{Koehne-Pruess-Wilke:Two-Phase-Navier-Stokes}
and \cite[Theorem~6.3]{Pruess-Simonett:Two-Phase-Navier-Stokes-Analytic};
see also~\cite{Escher-Pruess-Simonett:Parabolic-Analyticity, Escher-Pruess-Simonett:Stefan-Analytic}.

\section{Semiflow}
	\seclabel{semiflow}
\subsection{Proof of \Thmref{semiflow}}
That the solutions to \eqnref{model} obtained by \Thmref{well-posedness} generate a local semiflow in the phase manifold $\cS_p(\Omega)$
can be seen as follows.
Given
\begin{equation*}
	u_0 \in W^{2 - 2/p}_p(\Omega \setminus \Gamma_0), \quad \Gamma_0 \in W^{3 - 2/p}_p, \quad c_0 \in W^{2 - 2/p}_p(\Omega_+(0), \bR_+)
\end{equation*}
that satisfy the regularity and compatibility conditions as stated in \Thmref{well-posedness},
we obtain a local solution on some time interval $[0,\,a)$ with $a > 0$.
However, these solutions belong to the maximal regularity class defined in \Secref{linearization} and, thus,
admit a trace at $t = a$, which allows to reinvoke \Thmref{well-posedness} to continue the solution on a larger time interval.
Note that this also allows to choose a new reference manifold for the construction of the solution.
Repeating this procedure we obtain a maximal time interval $[0,\,a^\ast)$, which may be bounded by the fact that
\begin{equation*}
	\lim_{t \rightarrow a^\ast} (u(t),\,\Gamma(t),\,c(t))
\end{equation*}
does not exist in $\cS_p(\Omega)$ or by the fact that the interface undergoes a change in its topology or touches the outer boundary $\partial \Omega$,
in which case the model is no longer suitable to describe the situation.

\section*{Acknowledgements}
The first two authors gratefully acknowledge financial support by the Deutsche Forschungsgemeinschaft
within the priority programme ``Transport Processes at Fluidic Interfaces'', SPP 1506 (D.\,B.)
and the International Research Training Group ``Mathematical Fluid Dynamics'', IRTG 1529 (M.\,K.).

\bibliographystyle{plain}
\bibliography{references}
\end{document}